\documentclass{entcs}
\usepackage{entcsmacro}

\usepackage{QED}

\def\G{\Gamma}

\def\te{\theta} 
\def\l{\lambda}

\def\m{\mu}

\def\sou{\underline} 

\def\f{\rightarrow}

\def\<{\langle}
\def\>{\rangle}
\def\F{\displaystyle\frac}

\def\rd{\triangleright^*}

\def\sou{\underline}
\def\ou{\vee}
\def\et{\wedge}

\def\lastname{NOUR}
\begin{document}
\begin{frontmatter}
  \title{ A semantics of realisability for the classical
propositional natural deduction}
 \author{Karim NOUR}
  \address{Equipe de Logique, Universit\'e de Savoie\\
    73376 Le Bourget du Lac, France} \author{Khelifa SABER}
  \address{Equipe de Logique, Universit\'e de Savoie\\
    73376 Le Bourget du Lac, France}  \thanks[myemail]{Email:
    \href{mailto:knour@univ-savoie.fr} {\texttt{\normalshape
        knour@univ-savoie.fr}}} \thanks[coemail]{Email:
    \href{mailto:ksabe@univ-savoie.fr} {\texttt{\normalshape
        ksabe@univ-savoie.fr}}}

\begin{abstract} 
 In this paper, we introduce a semantics of realisability for the
 classical propositional natural deduction and we prove a correctness
 theorem. This allows to characterize the operational behaviour of
 some typed terms.
\end{abstract}
\end{frontmatter}

\section {Introduction}
 Natural deduction system is one of the main logical system which was
 introduced by Gentzen \cite{gen} to study the notion of proof. The
 full classical natural deduction system is well adapted for the human
 reasoning. By full we mean that all the connectives ($\f$, $\et$ and
 $\ou$) and $\perp$ (for the absurdity) are considered as primitive
 and they have their intuitionistic meaning. As usual, the negation is
 defined by $\neg A = A \f \perp$. Considering this logic from the
 computer science of view is interesting because, by the Curry-Howard
 correspondence, formulas can be seen as types for the functional
 programming languages and correct programs can be extracted. By this
  correspondence the corresponding calculus is an extension of the $\l
 \m$-calculus with product and co-product.

Until very recently (see the introduction of \cite{deG2} for a brief
history), no proof of the strong normalization of the cut-elimination
procedure was known for full logic. In \cite{deG2}, P. De Groote gives a
such proof for classical propositional natural deduction by using the
CPS-transformation. R. David and the first author give in \cite{dav2}
a direct and syntactical proof of this result. R. Matthes recently
found another semantical proof of this result (see \cite{matt}).

In order to prove the strong normalization of classical propositional
natural deduction, we introduce in \cite{Nour7} a variant of the
reducibility candidates, which was already present in
\cite{Par2}. This method has been introduced by J.Y. Girard. It consists in
associating to each type $A$ a set of terms $|A|$, such that every
term is in the interpretation of its type (this is called
``the adequation lemma''). To the best of our knowledge, we obtain the
shortest proof of this result.

In this paper, we define a semantics of realisability of classical
propositional natural deduction inspired by
\cite{Nour7} and we estabilish a correctness theorem. The idea 
is to replace the set of strongly normalizing terms used in the proof
presented in \cite{Nour7} by a set having the properties necessary to
keep the adequation lemma. This result allows to characterize the
operational behaviour of terms having some particular types.

 The paper is organized as follows. Section 2 is an introduction to
 the typed system and the relative cut-elimination procedure. In
 section 3, we define the semantics of realisability and we prove the
 correctness theorem. In section 4, we give some applications of this
 result.

\section{Notations and  definitions}
\begin{definition}
We use notations inspired by the paper \cite{and}. 
\begin{enumerate}
\item Let $\mathcal{X}$ and $\mathcal{A}$ be two disjoint alphabets
for distinguishing the $\lambda$-variables and $\mu$-variables
respectively.  We code deductions by using a set of terms
$\mathcal{T}$ which extends the $\l$-terms and is given by the
following grammars:
\begin{center}
$
\mathcal{T} \; := \;\mathcal{X} \; | \; 
\lambda\mathcal{X}.\mathcal{T}\;
|\; (\mathcal{T}\;\;\mathcal{E}) \; | \; \< \mathcal{T},\mathcal{T} \> 
\; | \;$$
\omega_1 \mathcal{T}$$ \; |\;$$ \omega_2 \mathcal{T}$$ \; | \;
\mu\mathcal{A}.\mathcal{T} \; | \; (\mathcal{A}\; \; \mathcal{T})
$

$
\mathcal{E} \; := \; \mathcal{T} \; | \; $$\pi_1$$ \; | \;$ $\pi_2 $$\; 
|
\; [\mathcal{X}.\mathcal{T},\mathcal{X}.\mathcal{T}]
$
\end{center}

An element of the set $\mathcal{E}$ is said to be an 
$\mathcal{E}$-term.

\item The meaning of the new constructors is given by the typing rules
below where $\G$ (resp. $\Delta$) is a context, i.e. a set of declarations of the form
$x : A$ (resp. $a :  A$) where $x$ is a $\l$-variable (resp. $a$
is a $\m$-variable) and $A$ is a formula.
 
\begin{center}
 $\F{}{\Gamma, x:A\,\, \vdash x:A\,\, ; \, \Delta}{ax}$
\end{center}

\begin{center}
$\F{\Gamma, x:A \vdash t:B;\Delta}{\Gamma \vdash \lambda x.t:A \to 
B;\Delta}{\to_i}
\quad\quad\quad 
\F{\Gamma \vdash u:A \to B;\Delta \quad \Gamma \vdash
v:A;\Delta}{\Gamma\vdash (u\; \; v):B;\Delta}{\to_e}$
\end{center}

\begin{center}
$\F{\Gamma \vdash u:A;\Delta \quad \Gamma \vdash v:B ; \Delta}{\Gamma
\vdash \<u,v\>:A \wedge B ; \Delta}{\wedge_i}$
\end{center}

\begin{center}
$\F{\Gamma \vdash t:A \wedge B ; \Delta}{\Gamma \vdash (t\;\;\pi_1):A ;
\Delta}{\wedge^1_e} 
\quad 
\F{\Gamma \vdash t:A\wedge B ; \Delta}{\Gamma \vdash (t\;\;\pi_2):B ; 
\Delta}{\wedge^2_e}$
\end{center}

\begin{center}
$\F{\Gamma\vdash t:A;\Delta}{\Gamma\vdash \omega_1
    t:A \vee B ;\Delta}{\vee^1_i}
\quad
\F{\Gamma \vdash t:B; \Delta}{\Gamma\vdash \omega_2 t:A\vee B 
;\Delta}{\vee^2_i}$
\end{center}

\begin{center}
$\F{\Gamma \vdash t:A\vee B ;\Delta\quad\Gamma, x:A \vdash u:C ;
\Delta\quad\Gamma, y:B \vdash v:C ; \Delta}{\Gamma \vdash (t\;\;[x.u, 
y.v]):C ; \Delta}{\vee_e}$ 
\end{center}

\begin{center}
$\F{\Gamma\vdash t:A ;\Delta, a:A}{\Gamma \vdash (a\;\;t):\bot ;\Delta,
a:A}{abs_i}
\quad
\F{\Gamma\vdash t:\bot; \Delta, a:A}{\Gamma \vdash\mu a.t:A; 
\Delta}{abs_e}$
\end{center}

\item The cut-elimination procedure corresponds to the  reduction rules 
given below. They are those we need to the subformula 
property.

\begin{itemize}
\item $(\lambda x.u \;\; v) \triangleright u[x:=v]$

\item $(\<t_1,t_2\>\;\;\pi_i) \triangleright t_i$

\item $(\omega_i t\;\;[x_1.u_1,x_2.u_2]) \triangleright  u_i[x_i:=t]$

\item $((t\;\;[x_1.u_1,x_2.u_2])\;\;\varepsilon) \triangleright
(t\;\;[x_1.(u_1\; \varepsilon),x_2.(u_2\;\varepsilon)])$

\item $(\m a.t\;\; \varepsilon) \triangleright \m
  a.t[a:=^*\varepsilon]$.
\end{itemize}

where $ t[a:=^*\varepsilon]$ is obtained from $t$ by replacing
inductively each subterm in the form $(a \; v)$ by $(a \; (v \; 
\varepsilon))$.

\item Let $t$ and $t'$ be $\mathcal{E}$-terms. The notation $t
\triangleright t'$ means that $t$ reduces to $t'$ by using one step
of the reduction rules given above. Similarly, $t \triangleright^* t'$
means that $t$ reduces to $t'$ by using some steps of the reduction
rules given above.
\end{enumerate}
\end{definition}

The following result is straightforward

\begin{theorem}(Subject reduction)
If $\G \vdash t : A ; \Delta$ and $t \triangleright^* t'$, then $\G 
\vdash t' : A ; \Delta$.
\end{theorem}

We have also the following properties (see \cite{and}, \cite{dav2}, \cite{deG2}, \cite{Nour7} and \cite{nour}).
\begin{theorem}(Confluence) If $t\triangleright^* t_1$ and $t\rd t_2$, 
then there exists $t_3$ such that $t_1\rd t_3$ and $t_2\rd t_3$.
\end{theorem}

\begin{theorem}(Strong normalization) If $\G \vdash t : A ; \Delta$, 
then $t$ is strongly normalizable.
\end{theorem}

\section {The semantics}

\begin{definition}
\begin{enumerate}
\item We denote by $\mathcal{E}^{<\omega}$  the set of finite sequences 
of $\mathcal{E}$-terms. The empty sequence is denoted by $\emptyset$.
\item  We denote by $\bar{w}$ the sequence $w_1 w_2...w_n$. If $\bar{w}=w_1 w_2...w_n$, then $(t \;\bar{w})$ is $t$ if $n=0$ and $((t \; w_1)\;w_2...w_n)$ if $n\neq 0$. The term $t[a:=^*\bar{w}]$ is the term obtained from $t$ by replacing
  inductively each subterm in the form $(a \; v)$ by $(a \; (v
  \;\bar{w}))$. 
\item A set of terms $S$ is said to be $\mu$-saturated iff:
\begin{itemize}
\item For each terms $u$ and  $v$, if $ u\in S$ 
and $ v \rd u$, then $ v\in S$.
\item For each $a \in \mathcal{A}$ and for each $t \in S$, $\mu a.t \in
 S$ and $(a\; t) \in S$.
\end{itemize}
\item Consider two sets of terms $K$, $L$ and a $\m$-saturated set $S$,  
we define new sets of terms:
\begin{itemize}
\item $K \to L =\{ t$ / $(t\; u) \in L,$ for each $ u \in K\}$.
\item $K \wedge L = \{t$ / $(t\;\pi_1) \in K$ and $(t\;\pi_2) \in L 
\}$.
\item $K \vee L = \{t$ / for each $u,v$: if (for each $r \in K$,$s
 \in L$: $u[x:=r]\in S$ and $ v[y:=s]\in S)$, then $(t\;[x.u,y.v])\in
 S\}$.
\end{itemize}
\item  Let $S$ be a $\mu$-saturated set and $\{ R_i\}_{ i\in I}$ subsets of terms such that   
$R_i = X_i \to S$ for certains $X_i \subseteq
\mathcal{E}^{<\omega}$. A model $\mathcal{M}$ $ = \langle S; \{R_i\}_{
i\in I}\rangle$ is the smallest set of subsets of terms containing $S$
and $R_i$ and closed under constructors $\to$, $\wedge$ and $\vee$.
\end{enumerate}
\end{definition}

\begin{lemma} Let $\mathcal{M} =\langle S; \{R_i\}_{ i\in I}\rangle$
be a model and  $G \in \mathcal{M}$.

There exists a set $X\subseteq \mathcal{E}^{<\omega} $ such that $G= X \to S$.
\end{lemma}

\begin{proof} By induction on $G$.
\begin{itemize}
\item  $G=S$: Take $X=\{\emptyset\}$, it is clear that
$S=\{\emptyset\} \to S$.  

\item  $G=G_1 \to G_2$: We have $G_2=X_2 \to
S$ for a certain set $X_2$. Take $X=\{u\;
\bar{v}$ / $u \in G_1, \bar{v}\in X_2\}$. We can easly check that $G = X \to S $.

\item  $G=G_1 \wedge G_2$: Similar to the previous case. 

\item  $G=G_1 \vee G_2$: Take $X=\{[x.u,y.v]$ / for each $r\in G_1$ and $s\in
G_2\;,\; u[x:=r] \in S$ and $ v[y:=s] \in S \}$. By definition $G = X
\to S$.
\end{itemize}
\end{proof}

\begin{definition} Let $\mathcal{M} = \langle
S; \{R_i\}_{ i\in I}\rangle$ be a model and $G \in \mathcal{M}$, we
define the set $G^\perp = \cup \{X$ / $G = X \to S \}$.
\end{definition}

\begin{lemma}
Let $\mathcal{M} = \langle S; \{R_i\}_{ i\in I}\rangle$ be a model and
$G \in \mathcal{M}$. 

We have $G= G^\perp \to S$ ($G^\perp$ is the greatest $X$ such that $G
= X \to S$).
\end{lemma}

\begin{proof}
This comes from the fact that: if, for every $j \in J$, $G=X_j \to S$,
then $G=\cup_{j \in J} X_j \to S$.
\end{proof}

\begin{definition} 
\begin{enumerate}
\item Let $\mathcal{M} = \langle S; \{R_i\}_{ i\in I}\rangle$ be a model.  An
$\mathcal{M}$-interpretation $I$ is an application from
the set of propositional variables to $\mathcal{M}$ which we extend
for any type as follows:
\begin{itemize}
\item $I(\perp)=S$
\item $I(A \to B)= I(A) \to I(B)$.
\item $I(A\wedge B)= I(A) \wedge I(B)$.
\item $I(A\vee B)= I(A) \vee I(B)$.
\end{itemize}
The set $\vert A \vert_{\mathcal{M}} =\cap \{ I(A)$ / $I$ an
$\mathcal{M}$-interpretation$\}$ is the interpretation of $A$ in
$\mathcal{M}$.
\item The set $\vert A \vert = \cap \{\vert A \vert_{\mathcal{M}}$ /
$\mathcal{M}$ a model$\}$ is the interpretation of $A$.
\end{enumerate}
\end{definition}

\begin{lemma}(Adequation lemma) \label{adq}  Let $\mathcal{M} =\langle  
S; \{R_i\}_{ i\in I} \rangle$ be a model, $I$ a
$\mathcal{M}$-interpretation, $\Gamma =\{x_i : A_i\}_{1\le i\le n}$, $\Delta =\{a_j : B_j\}_{1\le j\le m}$, $ u_i
\in I(A_i)$, $\bar{v_j} \in I(B_j)^\perp$.

If $\Gamma \vdash t:A ; \Delta$, then
$t[x_1:=u_1,...,x_n:=u_n,a_1:=^*\bar{v_1},...,a_m:=^*\bar{v_m}]\in I(A)$.
\end{lemma}

\begin{proof} 
Let us denote by $s'$ the term 

$s[x_1:=u_1,...,x_n:=u_n,a_1:=^*\bar{v_1},...,a_m:=^*\bar{v_m}]$. 

The proof is by induction on the derivation, we consider the last
rule:

\begin{enumerate}
\item ax, $\to_e$ and $\wedge_e$: Easy.

\item $\to_i$: In this case $t=\l x.u$ and $A= B\to C$ such that 
$\Gamma, x:B\vdash u: C\,\,\, ; \, \Delta$. By induction hypothesis,
$u'[x:=v] \in I(C)= I(C)^\perp \to S$ for each $v\in I(B)$, then
$(u'[x:=v]\; \bar{w}) \in S$ for each $ \bar{w} \in I(C)^\perp
$, hence $((\l x.u'\;v)\;\bar{w})\in S$ because $((\l x.u'\;
v)\;\bar{w}) \rd (u'[x:=v]\;\bar{w})$. Therefore $t'=\l x.u' \in I(B) \to
I(C) = I(A)$.

\item  $\wedge_i$ and $\vee_i^j$: A similar proof.

\item $\vee_e$: In this case $t=(t_1\;[x.u,y.v])$ with $(\Gamma \vdash 
t_1:B\vee C ; \Delta)$, $(\Gamma, x:B\vdash u: A ; \Delta)$ and
$(\Gamma, y:C\vdash v: A ; \Delta)$. Let $r \in I(B)$ and $s \in
I(C)$, by induction hypothesis, $t'_1 \in I(B) \vee I(C)$, $u'[x:=r]
\in I(A) $ and $v'[y:=s] \in I(A)$. Let $\bar{w} \in I(A)^\perp$, then
$(u'[x:=r]\; \bar{w}) \in S $ and $(v'[y:=s]\; \bar{w})
\in S$, hence 
$(t'_1\;[x.(u'\; \bar{w}),y.(v'\; \bar{w})]) \in S$, since
$((t'_1\;[x.u',y.v') ]\;\bar{w}) \rd (t'_1\;[x.(u'\;
\bar{w}),y.(v'\; \bar{w})])$ then $((t'_1\;[x.u',y.v') ]\;\bar{w})\in
S$.  Therefore $t'=(t'_1\;[x.u',y.v'])\in I(A)$.

\item $abs_e$: In this case $t=\mu a.t_1 $
 and $\Gamma \vdash t_1 : \perp\,\,\, ; \, \Delta' , a:A$. Let
 $\bar{v} \in I(A)^\perp$. It suffies to prove that $(\mu a.t'_1
 \;\bar{v})
\in S$. By induction hypothesis, $t'_1[a:=^*\bar{v}] \in I(\perp)=S$, then $\m 
a.t'_1[a:=^*\bar{v}]\in S$ and $(\m a.t'_1 
\; \bar{v})\in S$. 
\item $abs_i$: In this case $t= (a_j \;u)$ and $\Gamma \vdash  u : B_j ; \Delta', a_j:B_j$. We have to prove that $t' \in S$. 
By induction hypothesis $u' \in I(B_j)$, then $(u'\;\bar{v_j}) \in S$,
hence $t'=(a\;(u'\; \bar{v_j})) \in S$.
\end{enumerate}
\end{proof}

\begin{theorem}(Correctness theorem)\label{closed}
If $\vdash t:A$, then $t \in\vert A \vert$.
\end{theorem}

\begin{proof}
Immediately from the previous lemma.
\end{proof}

\section{The operational behaviors of some typed terms}

The following results are some applications of the correctness theorem.

\begin{definition}
Let $t$ be a term. We denote $M_t$ the smallest set containing $t$
such that: if $u \in M_t$ and $a \in {\cal A}$, then $\m a.u \in M_t$
and $(a \; u) \in M_t$. Each element of $M_t$ is denoted $\sou{\m}.t$.
For exemple, the term $\m a.\m b.(a \; (b \; (\m c. (a \; \m d.t))))$ is
denoted by $\sou{\m}.t$.
\end{definition}
In the next of the paper, the letter $P$ denotes a propositional
variable which represents an arbitrary type.

\subsection{Terms of type $\perp \to P $ ``Ex falso sequitur quodlibet''}

\begin{example}
 Let ${\cal T} =\l z.\m a.z$. We have  ${\cal T} : \perp \to P$ and for
every term $t$ and $\bar{u} \in \mathcal{T}^{<\omega}$, $(({\cal T}\; t) \;
\bar{u}) \rd \m a.t$.
\end{example}

\begin{remark} The term $({\cal T} \;t)$ modelizes an instruction like
${\tt exit}(t)$ (${\tt exit}$ is to be understood as in the C
programming language). In the reduction of a term, if the sub-term
$({\cal T} \;t)$ appears in head position (the term has the form
$(({\cal T}\; t) \; \bar{u})$), then after some reductions, we obtain
$t$ as result.
\end{remark}

The general operational behavior of terms of type $\perp \to P$ is
given in the following theorem:

\begin{theorem}
Let $T$ be a closed term of type $\perp \to P$, then for every term
$t$ and $\bar{u} \in \mathcal{E}^{<\omega}$, $((T\; t) \;
\bar{u}) \rd \sou {\m}. t$.
\end{theorem}

\begin{proof} 
Let $t$ be a term and $\bar{u}\in \mathcal{E}^{<\omega}$.  Take
$S=\{v$ / $v\rd \sou{\m}. t\}$ and $R=\{\bar{u}\}\to S$. It is clear
that $S$ is $\m$-saturated set and $t \in S$. Let $\mathcal{M}=\langle
S;R \rangle$ and $I$ an $\mathcal{M}$-interpretation such that
$I(P)=R$. By the theorem
\ref{closed}, we have $T \in S \to (\{\bar{u}\}\to S)$, then  $((T\;t)\;
\bar{u}) \in S$ and  $((T\; t) \;
\bar{u}) \rd \sou {\m}. t$.
\end{proof} 

\subsection{Terms of type $(\neg P \to P) \to P $ ``Pierce law''} 

\begin{example} Let ${\cal C}_1 =\l z.\m a.(a \, (z\,\, \l y.(a\, y)))$ and 

${\cal C}_2 =\l z.\mu
a.(a \,
(z\,\,(\l x.a (z\,\, \l y. (a\,x)))))$.

We have $\vdash {\cal C}_i : (\neg P \to P) \to P$ for $i \in \{1,2\}$.

Let  $u,v_1,v_2$ be terms and $\bar{t} \in
\mathcal{E}^{<\omega}$, we have :

$(({\cal C}_1\; u)\; \bar{t}) \rd \m a.a\;((u\,\te_1) \, \bar{t})$ and $(\te_1
\, v_1) \rd (a \, (v_1\; \bar{t}))$

and

$(({\cal C}_2 \;u)\; \bar{t}) \rd \mu a.((a \, ((u \, \te_1) \, \bar{t}))\,\bar{t})$,
$(\te_1\; v_1) \rd (a \,((u \, \te_2)\,\bar{t}))$ and 
$(\te_2\; v_2) \rd (a \, (v_1 \, \bar{t}))$.
\end{example}

\begin{remark}
The term ${\cal C}_1$ allows to modelizing the ${\tt Call/cc}$ instruction in
the Scheme functional programming language.
\end{remark}

The following theorem describes the general operational behavior of
terms with type $(\neg P \to P) \to P$.

\begin{theorem}
Let $T$ be a closed term of type $(\neg P \to P) \to P$, then for
every term $u$ and $\bar{t}
\in\mathcal{E}^{<\omega}$, there exist $m \in \mathbb{N}$ and terms
$\te_1,...,\te_m$ such that for every terms $v_1,...,v_m$, we
have:

$((T\; u)\; \bar{t}) \rd \sou{\m}. ((u\,\te_1)\, \bar{t})$

$(\te_i\; v_i)\rd \sou{\m}. ((u \, \te_{i+1})\, \bar{t})$ for every $1
\le i \le m-1 $

$(\te_m\; v_m) \rd \sou{\m}. (v_{i_0}\, \bar{t})$ for a certain  $1\le i_0 \le m$ 
\end{theorem}

\begin{proof}
Let $u$ be a $\l$-variable and $\bar{t} \in \mathcal{E}^{<\omega}$.
Take $S=\{t$ / $\exists m\ge 0, \exists \te_1,...,\te_m$ : $t \rd
\sou{\m}. ((u \;\te_1)\,\bar{t})$, $(\te_i\; v_i) \rd \sou{\m}. 
((u \,\te_{i+1})\,\bar{t})$ for every $1\le i \le m-1$ and $(\te_m\; v_m) \rd 
\sou{\m}. (v_{i_0} \bar{t})$ for a certain $1\le i_0 \le m \}$ and $R= \{\bar{t}\} \to S$. It 
is clear that $S$ is a $\mu$-saturated set.  Let $\mathcal{M} =
\langle S;R \rangle$ and an $\mathcal{M}$-interpretation $I$ such that
$I(P)= R$. By the theorem \ref{closed}, $T \in [(R \to S) \to R] \to
(\{\bar{t}\}
\to S)$. It is suffies to check that $u \in (R \to S) \to R$. For
this, we take $\te \in (R \to S)$ and we prove that $(u\; \te) \in R$
i.e. $((u\; \te)\;\bar{t}) \in S$. But by the definition of $S$, it
suffies to have $(\te\; v_i) \in S$, which is true since the terms
$v_i \in R$, because $(v_i\;\bar{t})\in S$.
\end{proof}

\subsection{Terms of  type $\neg P \vee P $ ``Tertium non datur''}
 
\begin{example}
Let ${\cal W} =\mu b.(b\, \omega_1 \mu a.(b\, \omega_2 \l y.(a\,y)))$. We
have $\vdash {\cal W} : \neg P \vee P$.

Let $x_1, x_2$ be $\l$-variables, $u_1, u_2,v$ terms and $\bar{t}
\in \mathcal{E}^{<\omega}$. We have:

$({\cal W} \,[x_1.u_1,x_2.u_2]) \rd \mu b.(b \, \,u_1\,[x_1:=\te_1^1])$

$(\te_1^1 \,\bar{t}) \rd \mu a.(b\, u_2\,[x_2:=\te_2 ^2])$

$(\te_2^2 \,v) \rd (a(v \,\bar{t}))$

where $\te_1^1=\mu a.(b\; (\omega_2 \l y.(a\;y) \;[x_1.u_1,x_2.u_2]))$
and $\te_2 ^2=\l y.(a\;(y\;\bar{t}))$.
\end{example}

\begin{remark}
The term ${\cal W}$ allows to modelizing the ${\tt try...with...}$ instruction in
the Caml programming language.
\end{remark}

The following theorem gives the behavior of all terms with type $\neg
P \vee P$.

\begin{theorem}
Let $T$ be a closed term of type $\neg P \vee P$, then for every $\l$-variables $x_1, x_2$  and terms
$u_1 ,u_2$ and $(\bar{t_n})_{n \ge 1}$ a sequence of
$\mathcal{E}^{<\omega}$, there exist $m \in \mathbb{N}$ and terms
$\te_1^i,...,\te_m^i$ $1\le i \le 2$ such that for all terms
$v_1,...,v_m$, we have:

$(T\,[x_1.u_1,x_2.u_2]) \rd \sou{\m}. u_i[x_i:=\te_1^i]$

$(\te_j^1 \, \bar{t_j}) \rd \sou{\m}. u_i[x_i:=\te_{j+1}^i]$ for all $1\le j
\le m-1 $

$(\te_j^2\; v_j) \rd \sou{\m}. u_i[x_i:=\te_{j+1}^i]$ for all $1\le
j \le m-1 $

$(\te_m^1 \bar{t_m}) \rd \sou{\m}.(v_p \;\bar{t_q})$ for a certain $1\le p \le m 
$ and a certain $1\le q \le m$

$(\te_m^2\; v_m) \rd \sou{\m}.(v_p \;\bar{t_q})$ for a certain $1\le p \le m$
and a certain $1\le q \le m$

\end{theorem}

\begin{proof} Let $u_1$, $u_2$ be terms and $(\bar{t_n})_{n \ge 1 
}$ a sequence of $\mathcal{E}^{<\omega}$.  Take then $S=\{t$ /
$\exists m\ge 0, \exists \te_1^i,...,\te_m^i$ $ 1\le i \le 2$ : $t \rd
\sou{\m}. u_i[x_i:=\te_1^i],$ $(\te_j^1\; \bar{t_j}) \rd
\sou{\m}. u_i[x_i:=\te_{j+1}^i]$ for all $1\le j
\le m-1$,
 $(\te_j^2\; v_j) \rd \sou{\m}. u_i[x_i:=\te_{j+1}^i]$ for all $1\le j
 \le m-1$, $(\te_m^1 \;\bar{t_m}) \rd \sou{\m}.v_p (\bar{t_q})$ for
 certain $(1\le p \le m$ and $1\le q \le m)$ and $(\te_m^2 \;v_m) \rd
 \sou{\m}.( v_p\; \bar{t_q})$ for certain $(1\le p \le m$ and $1\le
 q \le m) \}$.  $R=\{\bar{t_1},...,\bar{t_n}\} \to S $.  By
 definition $S$ is a $\mu$-saturated set. Let $\mathcal{M}=\langle
 S;R\rangle$ and an $\mathcal{M}$-interpretation $I$ such that $I(P)=
 R$.  By the theorem \ref{closed}, $T \in [R \to S] \vee R$.  Let $\te
 \in R$, then, for all $i$, $(\te\; \bar{t_i}) \in S$. Let $\te' \in R
 \to S$, hence $(\te'\; v_i) \in S$ since $v_i \in R$ (because $(v_i\;
 \bar{t_i}) \in S$), therefore $(T\; [x_1.u_1,x_2.u_2]) \in S$.
\end{proof}

\end{document}